\documentclass[11pt]{amsart}
\usepackage{amssymb, latexsym, amsmath, amsfonts, tikz}
\usepackage{hyperref, enumitem, fancyhdr}
\usepackage{color}

\newtheorem{thm}{Theorem}[section]

\newtheorem{lem}[thm]{Lemma}
\newtheorem{prop}[thm]{Proposition}
\theoremstyle{definition}

\theoremstyle{remark}
\newtheorem{rem}[thm]{Remark}
\numberwithin{equation}{section}
\theoremstyle{remark}
\newtheorem{exam}[thm]{Example}

\newcommand{\mbb}{\mathbb}

\newcommand{\sm}{\setminus}

\newcommand{\no}{\noindent}

\newcommand{\cal}{\mathcal}
\newcommand{\ti}{\tilde}

\newcommand{\de}{\delta}

\begin{document}
\title{Rigidity Theorems for H\'{e}non maps-II}
%%\keywords{}
%%\thanks{The author was supported in part by a UGC--CAS Grant}
\author{Sayani Bera}

\address{SB: School of Mathematics, Ramakrishna Mission Vivekananda Educational and Research Institute, PO Belur Math, Dist. Howrah, 
West Bengal 711202, India}
\email{sayanibera2016@gmail.com}

\begin{abstract}
The purpose of this note is to explore further the rigidity properties of H\'{e}non maps from \cite{RigidityPaper}. For instance, we show that if $H$ and $F$ are H\'{e}non maps with the same Green measure ($\mu_H=\mu_F$), or the same filled Julia set ($K_H=K_F$), or the same Green function ($G_H=G_F$), then $H^2$ and $F^2$ have to commute. This in turn, gives that $H$ and $F$ have the same non--escaping sets. Further we prove that, either of the association of a H\'{e}non map $H$ to its Green measure $\mu_H$ or to its filled Julia set $K_H$ or to its Green function $G_H$ is locally injective.
\end{abstract}
\maketitle
\section{Introduction}
\no We continue to explore the rigidity properties of H\'{e}non maps from \cite{RigidityPaper}. The main motivation for \cite{RigidityPaper} was a result from the dynamics of polynomial maps in one variable by Beardon \cite{Be}, namely if two polynomials $P$ and $Q$ of degree greater than or equal to $2$, have the same Julia set, i.e., $J_P=J_Q$ then
\[ P \circ Q=\sigma \circ Q \circ P\]
where $\sigma(z)=az+b$ with $|a|=1$ and $\sigma(J_P)=J_P.$ In \cite{RigidityPaper}, we provide an analogue of this result for \textit{H\'{e}non maps} in $\mbb C^2.$ To explain it in detail, let us revisit the notations first. Let $\cal{H}$ denote the collection of maps $H$ defined as:
\begin{equation}\label{1.1}
H = H_m \circ H_{m-1} \circ \cdots \circ H_1
\end{equation}
where each $H_j$ is a map of the form
\begin{equation}\label{1.2}
H_j(x, y) = (b_jy+c_j, p_j(y) - \delta_j x) 
\end{equation}
with $p_j$ a polynomial of degree $d_j \ge 2$, $b_j\delta_j\neq 0$ and $c_j \in \mbb C.$ The degree of $H$ is $d = d_1d_2 \ldots d_m$. The phrase {\it H\'{e}non map}, hereafter will be used to refer a map as in (\ref{1.1}). The definition of H\'{e}non map deviates from the standard definition of H\'{e}non map in \textit{normal form}, as we allow the constants $b_j \in \mbb C^*$ and $c_j \in \mbb C.$ The reason, we work with this definition is the presence of this constants do not make any difference from the point of view of dynamics.

\medskip\no Recall that for a H\'{e}non map $H$, the non--escaping sets or the \textit{filled} positive and negative Julia sets of $H$ is defined as:
\[ K_H^\pm =\{ z \in \mbb C^2: H^{\pm n}(z) \text{ is bounded for every }n \ge 1\}.\]
Theorem \textbf{1.1} of \cite{RigidityPaper} says that, if an automorphism $F$ preserves the non--escaping sets of a H\'{e}non map, i.e., $F(K_H^\pm)=K_H^\pm$ then $F$ or $F^{-1}$ is a H\'{e}non map of the form (\ref{1.1}) and further, they commute upto a linear map. In particular,
\begin{align}\label{theorem 1.1}
 F \circ H=C \circ H \circ F  \text{ or }F^{-1} \circ H=C \circ H \circ F^{-1}
\end{align}
where $C(x,y)=(\delta_+ x, \delta_- y)$ with $|\de_\pm|=1.$ Though in \cite{RigidityPaper} we did not explicit mention about the presence of the constants $c_j$'s in $H$, but there presence do not significantly change any computation (see Lemma \ref{normal theorem}). The proof of this result crucially relied on techniques developed by Buzzard and Forn\ae ss in \cite{RootBuzzard}, the rigidity theorem of Dinh and Sibony in \cite{Dinh-Sibony}, and introducing the notion of B\"{o}ttcher coordinates for H\'{e}non maps of the form (\ref{1.1}), inspired from the construction of Hubbard and Oberste-Vorth in \cite{Hu-OV}.   

\medskip\no The goal of this paper is to improve the rigidity phenomenon for H\'{e}non maps further. We briefly recall the dynamical objects associated to a H\'{e}non map $H \in \cal{H}$ from \cite{BS} and \cite{BS1}. For $R > 0$, let 
\begin{align*}
V^+_R &= \{ (x,y) \in \mathbb C^2: \vert x \vert < \vert y \vert, \vert y \vert > R \},\\
V^-_R &= \{ (x,y) \in \mathbb C^2: \vert y \vert < \vert x \vert, \vert x \vert > R \},\\
V_R &= \{ (x, y) \in \mathbb C^2: \vert x \vert, \vert y \vert \le R \}.
\end{align*}
For a given $H \in \mathcal{H}$ there exists $R > 0$ such that
\[
H(V^+_R) \subset V^+_R, \; H(V^+_R \cup V_R) \subset V^+_R \cup V_R
\]
and
\[
H^{-1}(V^-_R) \subset V^-_R, \; H^{-1}(V^-_R \cup V_R) \subset V^-_R \cup V_R.
\]
Also, $K^{\pm} \subset V_R \cup V^{\mp}_R$ and
\begin{equation}\label{escape}
\mathbb C^2 \sm K^{\pm}_H = \bigcup_{n=0}^{\infty} (H^{\mp n})(V^{\pm}_R)
\end{equation}
The positive and negative Green functions associated to the H\'{e}non map $H$ is defined as:
\[
G^{\pm}_H(x, y) = \lim_{n \rightarrow \infty} \frac{1}{d^n} \log^+ \Vert H^{\pm n}(x, y) \Vert.
\]
The Green functions $G_H^\pm$ is pluri--subharmonic on $\mbb C^2$ and non--negative everywhere, pluri--harmonic on $\mbb C^2 \sm K^{\pm}_H$ and vanish precisely on $K^{\pm}_H$. By construction the Green functions satisfy
\[ G_H^\pm \circ H(z)=d G_H^\pm(z)\]
and they have a logarithmic growth near infinity, i.e., there exists $R>0$, sufficiently large such that for $(x,y) \in V_R \cup V_R^+$
\[ G_H^+(x,y)=\log|y|+O(1)\]
and for $(x,y) \in V_R \cup V_R^-$
\[ G_H^-(x,y)=\log|x|+O(1).\]  
The Green function of $H$ is defined as $G_H=\max\{G_H^+,G_H^-\}.$ Then $G_H$ is a pluri--subharmonic function with logarithmic growth in  $\mbb C^2$. The \textit{forward} and the \textit{backward Julia sets} associated to $H$ is defined as:
\[J^{\pm}_H = \partial K^{\pm}_H.\]
Further we define the sets $J_H$ and $K_H$ as follows:
\[ J_H=J_H^+ \cap J_H^- \text{ and } K_H=K_H^+ \cap K_H^-.\] 
Henceforth, by \textit{Julia set} we will mean the set $J_H$ and by \textit{filled Julia set} we will mean the set $K_H.$ It turns out that both $J_H$ and $K_H$ are compact sets, which are completely invariant under $H.$ Also, $G_H^\pm$ are the pluri--complex Green functions for $K^{\pm}_H$ respectively and $G_H$ the pluri--complex Green function of $K_H$. The supports of the positive closed $(1,1)$ currents 
\[
\mu^{\pm}_H = \frac{1}{2 \pi}dd^c G^{\pm}_H
\]
are $J^{\pm}_H.$ The \textit{Green measure} of $H$ is defined as:
$$\mu_H = \mu^+_H \wedge \mu^-_H=\Big(\frac{1}{2 \pi}dd^c G_H\Big)^2.$$ Also, $\mu_H$ is the equillibrium measure of $K_H$ which is invariant under $H$ and $H^{-1}.$
Thus for two H\'{e}non maps $H$ and $F$, if $K_H=K_F$ then $\mu_H=\mu_F.$ 

\medskip\no In Section \ref{2}, we consider the simplest form of a H\'{e}non map, i.e., $H(x,y)=(y,p(y)-a x)$ and prove the following rigidity result:
\begin{thm}\label{commute}
Suppose $H(x,y)=(y,p(y))-ax)$ and $F(x,y)=(y,q(y)-bx)$ be two (simple) H\'{e}non maps. Then each of the following are true.
\begin{itemize}
\item[(i)] If $F(K_H^\pm)=K_H^\pm,$ then $F=H.$
\item [(ii)]If $\mu_H=\mu_F,$ then $F=H.$
\item[(iii)] If $K_H=K_F,$ then $F=H.$
\item[(iv)]If $G_H=G_F,$ then $F=H.$
\end{itemize} 
\end{thm}
\no Note that Theorem \textbf{\ref{commute}} is an analogue of Beardon's result (\cite{Be}) for simple H\'{e}non maps. To prove this, we use facts about the uniqueness of pluri--subharmonic functions with logarithmic growth from \cite{BedfordTaylor} and conclude that any two H\'{e}non maps with the same Green measure have the same Green function. Next, we explore the dynamical properties of a H\'{e}non map in appropriate regions and prove  that the positive and negative Green functions of these H\'{e}non maps actually coincide on $V_R^\pm$, respectively for appropriately chosen large $R.$ Finally, by appealing to techniques similar to the proof of Theorem \textbf{1.1} in \cite{RigidityPaper} the result follows. 

\medskip\no In Section \ref{3}, we first establish a rigidity result similar to Theorem \textbf{1.1} from \cite{RigidityPaper}, which is stated as follows:
\begin{thm}\label{improved_relation}
Let $F$ be a non--linear automorphism that preserves $K_H^\pm$ where $H$ is a H\'{e}non map then $F$ or $F^{-1}$ is a H\'{e}non map and $$F^2 \circ H^2=H^2 \circ F^2.$$
\end{thm}
\no To prove this, without loss of generality first we assume that origin is fixed by the H\'{e}non map $H$, and prove that $H^2$ can be uniquely expressed (upto composition by linear maps) in the \textit{normal form} such that
\[ H^2=H_m \circ \cdots \circ H_1\]
where $H_i(x,y)=(y,p_i(y)-\de_i x)$ and $H_i(0)=0$ for every $1 \le i \le m.$ Note that as a consequence of Theorem \textbf{1.1} from \cite{RigidityPaper}, $F$ and $H$ satisfy (\ref{theorem 1.1}). Now by further analyzing closely the behavior of the linear map $C$, it is possible to remove the linear map in the second iterate of $H$ and $F.$ In this context, we also provide an explicit example to emphasize that $H$ and $F$ might fail to commute even though their squares commute. Further appealing to ideas from Section \ref{2}, we prove a version of Theorem \textbf{\ref{commute}}, particularly for H\'{e}non maps, which is stated as follows:
\begin{thm}\label{KHpm}
Let $F$ and $H$ are H\'{e}non maps such that either $K_H=K_F$ or $\mu_H=\mu_F$ or $G_H=G_F.$ 
Then $K_H^\pm=K_F^\pm$ and $G_H^\pm=G_F^\pm.$
\end{thm}
\no Observe that Theorem \textbf{\ref{improved_relation}} applied to Theorem \textbf{\textbf{\ref{KHpm}}}, gives that $H^2$ and $F^2$ actually commutes. To mention here, Lamy in \cite{Lamy} proved that if two H\'{e}non maps $H$ and $F$ have the same positive Green function, i.e., $G_H^+=G_F^+$ then there exist integers $m$ and $n$ such that $F^m=H^n.$ Note that Theorem \textbf{\ref{KHpm}} improves this result, in the sense if $H$ and $F$ are H\'{e}non maps such that $G_H=G_F$, (or $K_H=K_F$, or $\mu_H=\mu_F$) then there exist positive integers $m$ and $n$ such that $F^m=H^n.$ Also, Dujardin and Favre in connection to their work on Manin--Mumford problem for plane polynomial automorphisms, in \cite{DF}, used number theoretic techniques to show that, if $F$ and $G$ are polynomial automorphisms of H\'{e}non type of the affine plane over a number field that share a Zariski dense set of periodic points, then there exist positive integers $m$ and $n$ such that $F^m=G^n.$

\medskip\no Finally, we consider the space of H\'{e}non maps $\cal{H}$ with the topology of uniform convergence over compact sets and consider either of the associations -- $H$ to its Green measure or $H$ to its Green function or $H$ to its filled Julia set, i.e., $H \to G_H$ or $H \to \mu_H$ or $H \to K_H$ on $\cal{H}.$ By using Theorem \textbf{\ref{improved_relation}} and the fact that, roots of a H\'{e}non map are finite from \cite{RootBuzzard} we prove that this association is locally injective in either of the cases. The statement of the theorem is stated as:
\begin{thm}\label{injective}
Let $\mathcal{H}$ denote the space H\'{e}non maps then the mappings $H \to G_H$, $H \to \mu_H$ and $H \to K_H$ is locally injective on $\mathcal{H}.$
\end{thm}
\subsection*{Acknowledgements} The author would like to thank Ratna Pal and Kaushal Verma for introducing her to the problem. 
\section{Proof of Theorem \ref{commute}}\label{2}
\no In this section, we first prove a development to Theorem \textbf{1.1} from \cite{RigidityPaper}.
\begin{thm}\label{relation}
Let $H$and $F$ be  H\'{e}non maps of the form (\ref{1.1}) such that $F$ preserves $K_H^\pm$ ($F(K_H^\pm)=K_H^\pm$) then there exists an $\eta$ such that $|\eta|=1$ and
\[ F\circ H=C_\eta \circ H \circ F=H \circ F \circ C_{\eta} \text{ or }F\circ H=C_\eta \circ H \circ F=H \circ F \circ C_{\eta}^{-1}.\]
where $C_\eta(x,y)=(\eta x,\eta^{-1} y).$ %and $\eta p(\eta  y)=p(y).$
\end{thm}
\begin{proof}
%From Theorem \textbf{1.1} of the \cite{RigidityPaper}, we know that $f$ should be a polynomial map, in particular either $f$ or $f^{-1}$ is a H\'{e}non map or $f$ is an affine map.
%
%\medskip\no \textit{Case 1: }Suppose $f$ is a H\'{e}non map of degree $d_f \ge 2$, then 
%
%\medskip\no Also by Theorem \textbf{1.1} in \cite{RigidityPaper} there exists a H\'{e}non map $F$ and $r_f,r_H>0$ such that
%\[ f=\sigma_f F^{r_f} \text{ and } H=\sigma_H F^{r_H}\]
%where $\sigma_f$ and $\sigma_H$ are linear maps that preserve $K_H^\pm.$ Since $H$ is H\'{e}non map of the form (\ref{1.2}), simply by comparing the degree in the above equation we have that $r_H=1$, i.e., $F$ is also of the form (\ref{1.2}).
Let $d_H$ denote the degree of $H.$ From Proposition \textbf{2.1} of \cite{RigidityPaper}, there exists $R>0$ (sufficiently large), appropriate non--zero constants $c_H,c'_H,c_F,c'_F \in \mbb C^*$ and non--vanishing holomorphic functions (B\"{o}ttcher coordinates)
$\phi_H^\pm, \phi_F^\pm: V_R^\pm \rightarrow \mathbb{C}$  such that 
\begin{align}\label{Bot1}
\phi_H^+\circ H(x,y)=c_H{(\phi_H(x,y))}^{d_H} \text { and }\phi_F^+\circ F(x,y)=c_F{(\phi_F(x,y))}^{d_F}
\end{align}
in $V_R^+$ and 
\begin{align}\label{Bot2}
\phi_H^-\circ H^{-1}(x,y)=c_H'{(\phi_H^-(x,y))}^{d_H} \text { and }\phi_F^-\circ F^{-1}(x,y)=c_F'{(\phi_F(x,y))}^{d_F}
\end{align}
in $V_R^-$. Further,
\[
\phi_F^+(x,y),\phi_H^+(x,y)\sim y \text{ as } \lVert(x,y)\rVert\rightarrow \infty \text{ in } V_R^+
\]
and
\[
\phi_F^-(x,y), \phi_H^-(x,y)\sim x \text{ as } \lVert(x,y)\rVert\rightarrow \infty \text{ in } V_R^-.
\]
Also, from the proof of Theorem \textbf{1.1} in \cite{RigidityPaper} the positive and negative Green function of $F$ and $H$ coincides, i.e., $G_H^\pm=G_F^\pm$ and 
\begin{equation}\label{Green1}
G_H^+=\log \lvert\phi_H^+\rvert+\frac{1}{d_H-1}\log \lvert c_H\rvert=\log \lvert\phi_F^+\rvert+\frac{1}{d_F-1}\log \lvert c_F\rvert=G_F^+
\end{equation}
in $V_R^+$,
\begin{equation}\label{Green2}
G_H^-=\log \lvert\phi_H^-\rvert+\frac{1}{d_H-1}\log \left\lvert {c_H'}\right\rvert=\log \lvert\phi_F^-\rvert+\frac{1}{d_F-1}\log \lvert c'_F\rvert=G_F^-
\end{equation}
in $V_R^-.$ Since $\phi_H^+$ and $\phi_F^+$ are both asymptotic to $y$ as $\lVert(x,y)\rVert\rightarrow \infty$ in $V_R^+$, it follows that 
\begin{equation}\label{const}
\frac{1}{d_H-1}\log \lvert c_H \rvert=\frac{1}{d_F-1}\log \lvert c_F \rvert
\end{equation}
and consequently 
\begin{equation*}\label{Phi+}
\phi_H^+ \equiv \phi_F^+
\end{equation*}
in $V_R^+$. Similarly,
\begin{equation*}\label{Phi-}
\phi_H^- \equiv \phi_F^-
\end{equation*}
in $V_R^-$. From now on we shall write $\phi^\pm$ for $\phi_H^\pm \equiv \phi_F^\pm$.

\medskip 
\no 
By (\ref{const}),
\[
c_H^{d_F}c_F=c_F^{d_H}c_H\delta_1 \text{ and }{c'_H}^{d_F}{c'_F}={c'_F}^{d_H}{c'_H}\delta_2 
\]
for some $\delta_i$ ($i=1,2$) with $\lvert \delta_i\rvert=1$.

\medskip\no \textit{Step 1: } There exist $\delta_1, \tilde{\delta}_1 \in \mbb C^*$ such that in $\mbb C^2$
\begin{align}\label{Step 1}
  \pi_2 \circ F \circ H=\delta_1 (\pi_2 \circ H \circ F) \text{ and } \pi_2 \circ F^{-1} \circ H^{-1}=\tilde{\delta}_1(\pi_2 \circ H^{-1} \circ F^{-1})
\end{align} 
where $|\delta_1|$ and $\tilde{\delta}_1|=1.$

\medskip\no
By (\ref{Bot1}),  
\begin{equation}\label{FoH}
\phi^+ \circ F \circ H(x,y)=c_F{(\phi^+ \circ H (x,y))}^{d_F}=c_F c_H^{d_F}{(\phi^+(x,y))}^{d_H d_F}
\end{equation}
and similarly,
\begin{equation}\label{HoF}
\phi^+ \circ H \circ F(x,y)=c_H{(\phi^+ \circ F (x,y))}^{d_H}=c_H c_F^{d_H}{(\phi^+(x,y))}^{d_H d_F}.
\end{equation}
Therefore, 
\begin{equation*}\label{equality}
\phi^+ ( F \circ H) =\delta_1 \phi^+ ( H \circ F)
\end{equation*}
on $V_R^+$. 
Since
\begin{equation*}
\phi^+ \circ F \circ H(x,y)\sim { \pi_2 \circ F\circ H}(x,y)  
\end{equation*}
and 
\begin{equation*}
\phi^+ \circ H \circ F(x,y)\sim {\pi_2 \circ H\circ F}(x,y)  
\end{equation*}
as $\lVert(x,y)\rVert\rightarrow \infty$ in $V_R^+$, it follows that for a fixed $x_0 \in \mathbb{C}$, 
\[
{\pi_2 \circ F\circ H}(x_0,y)- \delta_1{(\pi_2 \circ H\circ F)}(x_0,y)\sim 0 \text{ as } \lvert y \rvert \rightarrow \infty.
\]The expression on the left is a polynomial in $y$ and hence
\[
{ \pi_2 \circ F\circ H}(x_0,y)=\delta_1{(\pi_2 \circ H\circ F)}(x_0,y)
\]
for all $y\in \mathbb{C}$. Therefore,
\begin{equation}\label{rel1}
{ \pi_2 \circ F\circ H} \equiv \delta_1{ (\pi_2 \circ H\circ F)}
\end{equation}
Now note that $F\circ H(V_R^+) \cap H\circ F(V_R^+) \neq \phi$, since $[0:1:0]$ is an attracting fixed point of both $F$ and $H$ in $\mbb P^k.$ Let $\cal{U}=F\circ H(V_R^+) \cap H\circ F(V_R^+)$ (an open subset $V_R^+$). From (\ref{FoH}) and (\ref{HoF}) it follows that for $(x,y) \in \cal{U}$
\begin{align*}
\phi^+ \circ F \circ H \circ H^{-1} \circ F^{-1}(x,y)=c_F c_H^{d_F}{(\phi^+(H^{-1} \circ F^{-1}(x,y)))}^{d_H d_F} 
\end{align*}
\begin{align*}
\phi^+ (x,y)=c_F c_H^{d_F}{(\phi^+ \circ H^{-1} \circ F^{-1}(x,y))}^{d_H d_F}
\end{align*}
and similarly,
\begin{equation*}
\phi^+(x,y)=c_H c_F^{d_H}{(\phi^+ \circ F^{-1} \circ H^{-1}(x,y))}^{d_H d_F}.
\end{equation*}
Hence
\[
{(\phi^+ \circ F^{-1}\circ H^{-1} (x,y))}^{d_H d_F} = \eta {(\phi^+ \circ H^{-1} \circ F^{-1} (x,y))}^{d_H d_F}
\]
on $\cal{U}$ with $\vert \eta \vert = 1$. Consequently, there exists $\tilde{\delta}_1$ (an appropriate $d_H d_F-$th root of $\eta$) such that
\[
\phi^+\circ (F^{-1}\circ H^{-1})=\tilde{\delta}_1 \phi^+ \circ (H^{-1}\circ F^{-1})
\]
on $\cal U$. Note that $\vert \tilde{\delta}_1 \vert = 1$. Pick a point $c \in \mbb C^*$ and a sequence $\{y_n\} \in \mbb C$ such that $|y_n| \to \infty.$ The points
\[ [c:y_n:1] \to [0:1:0] \text{ in } \mbb P^2\]
as $n \to \infty,$ i.e., $(c,y_n) \in \cal{U}$ for $n$ sufficiently large. Note that $(F^{-1}\circ H^{-1})(c,y_n)$, and $ (H^{-1}\circ F^{-1})(c,y_n)$, is contained in $V_R^+$ as a consequence of  
$(c,y_n)$ being contained in $\mathcal{U}$. Hence we have,  
\[
{\pi_2 \circ F^{-1}\circ H^{-1}}_1(c,y_n), \;{\pi_2 \circ H^{-1}\circ F^{-1}}(c,y_n)\rightarrow \infty
\]
as $n\rightarrow \infty$. 

\medskip
\no 
Since $\phi^+(x,y)\sim y$ as $\lVert(x,y)\rVert\rightarrow \infty$,  
\begin{equation*}
{\pi_2 \circ F^{-1}\circ H^{-1}}(c,y_n)-\tilde{\delta}_1{(\pi_2 \circ H^{-1}\circ F^{-1})}(c,y_n)\rightarrow 0
\end{equation*} 
as $n\rightarrow \infty$. The expression on the left is a polynomial in $y$ for each fixed $c$ and thus
\begin{equation*}
{\pi_2 \circ F^{-1}\circ H^{-1}}(c,y)=\tilde{\delta}_1{(\pi_2 \circ H^{-1}\circ F^{-1})}(c,y)
\end{equation*}
for all $y\in \mathbb{C}$. Using the same argument as in the previous case, we get
\begin{equation}\label{rel2}
{\pi_2 \circ F^{-1}\circ H^{-1}}\equiv \tilde{\delta}_1{(\pi_2 \circ H^{-1}\circ F^{-1})}.
\end{equation}
Thus \textit{Step 1} is complete from (\ref{rel1}) and (\ref{rel2}).

\medskip\no \textit{Step 2: } There exist $\delta_2, \tilde{\delta}_2 \in \mbb C^*$ such that in $\mbb C^2$
\begin{align}\label{Step 2}
\pi_1 \circ F \circ H=\delta_2 (\pi_1 \circ H \circ F) \text{ and } \pi_1 \circ F^{-1} \circ H^{-1}=\tilde{\delta}_2(\pi_1 \circ H^{-1} \circ F^{-1})
\end{align} 
where $|\delta_2|$ and $\tilde{\delta}_2|=1.$

\medskip\no Using similar idea as in \textit{Step 1}, by interchanging the role of $F$ with $F^{-1}$ and $H$ with $H^{-1}$ and working with the function $\phi^{-1}$ in $V_R^-$, \textit{Step 2} follows.

\medskip\no Thus from (\ref{Step 1}) and (\ref{Step 2}) we have that 
\begin{align}\label{FoH and HoF}
F \circ H=C_1 \circ H \circ F  \text{ and } F^{-1} \circ H^{-1}=C_2 \circ H^{-1} \circ F^{-1} 
\end{align}
where 
$C_1(x,y)=(\delta_2 x, \delta_1 y)$ and $C_2(x,y)=(\tilde{\delta}_2 x, \tilde{\delta}_1 y). $
Now (\ref{FoH and HoF}) gives
\begin{align}\label{main 1}
F \circ H=C_2 \circ H \circ F=H \circ F \circ C_1.
\end{align}

\no\textit{Step 3: } $C_1(x,y)=C_2(x,y)$ or $C_1(x,y)=C_2^{-1}(x,y)$ and $C_1(x,y)=C(x,y)=(\eta x, \eta^{-1} y).$
 
 \medskip
\no Let $A=DH(F(0))DF(0).$ Then by chain rule applied to (\ref{main 1})
\begin{align}\label{3.1}
 D_2 A=A D_1
\end{align}
where $D_2=\text{Diag}(\tilde{\delta}_2,\tilde{\delta}_1)$ and $D_2=\text{Diag}({\delta}_2,{\delta}_1).$ Since $A$ is invertible, one of the diagonal element and one of the off--diagonal element cannot be simultaneously zero. Hence this leads to two generic situation: 
%Then from (\ref{3.1})
%\[ \begin{pmatrix}
%\tilde{\delta}_2 a &\tilde{\delta}_2b \\
%\tilde{\delta}_1 c &\tilde{\delta}_1 d
%\end{pmatrix}
%=\begin{pmatrix}
%{\delta}_2 a &{\delta}_1b \\
%{\delta}_2 c &{\delta}_1 d
%\end{pmatrix}
% \]
%Since $F=(\ti{b}y, q(y)-\ti{\de}x)$ and $H(by,p(y)-\delta x)$, 
%\[ DF(0)=\begin{pmatrix}
%0  &\ti{b} \\
%-\ti{\de}  & q'(0)
%\end{pmatrix}
%\text{ and }
%DH(F(0))=\begin{pmatrix}
%0  &{b} \\
%-{\de}  & p'(q(0))
%\end{pmatrix}.
%\]
%Therefore \[A=\begin{pmatrix}
%-b\ti{\de} & bq'(0) \\
%-\ti{\de}p'(q(0)) & -b\ti{\de}+p'(q(0))q'(0)
%\end{pmatrix}.\]
%Note that $-b\ti{\de} \neq 0$, hence from (\ref{3.1}) we have $\ti{\de}_2=\de_2.$ Since $\text{det} A \neq 0$,
\begin{itemize}
\item[(i)]Either both the diagonal elements should be non--zero. This gives $\ti{\de}_1=\de_1$  and $\ti{\de}_2=\de_2$(from (\ref{3.1})).

\medskip
\item[(ii)] Or both off--diagonal elements should be non--zero. This gives $\de_2=\ti{\de}_1$ and $\ti{\de}_2=\de_1$ (from (\ref{3.1})).
\end{itemize}
Further, as H\'{e}non maps have a constant Jacobian it follows (\ref{main 1}) that 
\[ \text{det }C_1=\text{det }C_2=1.\]Hence \[\de_1=\ti{\de}_1=\eta \text{ and } \de_2=\ti{\de}_2=\eta^{-1} \text{ or }\de_1=\ti{\de}_2=\eta \text{ and } \de_2=\ti{\de}_1=\eta^{-1}.\] Thus the proof.
\end{proof}
\no Let $H$ be a H\'{e}non map of the form(\ref{1.1}), i.e., 
\[ H=H_m \circ \cdots \circ H_1\]
where $H_i$ are H\'{e}non maps of the form (\ref{1.2}). Recall that, there exists $R_H>0$ for which $V^\pm$ can be defined as: 
\begin{align*}
V^+&=V^+_{R_H}=\{ (x,y) \in \mbb C^2 : |y|>\max\{|x|,R\}\} \text{ and }\\ 
V^-&=V^-_{R_H}=\{ (x,y) \in \mbb C^2 : |x|>\max\{|y|,R\}\}
\end{align*} 
such that $H_i(V^+) \subset V^+$, $H_i^{-1}(V^-) \subset V^-$ for every $1 \le i \le n.$ For $R_0>R_H$ consider the region $D^1(R_0) \subset V^-$ and $D^2(R_0) \subset V^+$ defined as:
\[D^1(R_0)=\{(x,y) \in V^-: |y|< {R_H} \text{ and }|x|>R_0\}\] and 
\[D^2(R_0)=\{(x,y) \in V^+:|x|< {R_H} \text{ and }|y|>R_0\}.\]
%Here $D_{R_0}$ is the disc of radius $R_0$ at the origin in $\mbb C.$
\begin{figure}[ht!]
\includegraphics[scale=0.4]{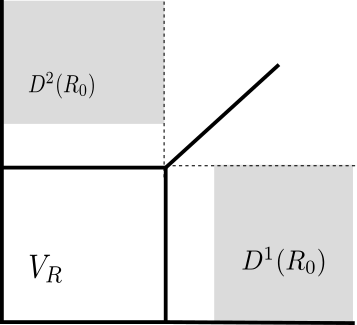}
\caption{The sets $D^1(R_0)$ and $D^2(R_0)$ in $\mbb C^2$}
\end{figure}
\begin{prop}\label{Green-main}
Let $H$ be a H\'{e}non map of the form (\ref{1.1}), i.e.,
\[ H=H_m \circ \cdots \circ H_1\]
where $H_i$'s are of the form (\ref{1.1}) for every $1 \le i \le m.$ Then there exist $R_H>0$ and $R_0 \gg R_H$ such that 
\begin{itemize}
\item[(i)] on $D^2(R_0)$, $G_H^-< G_H^+.$
\item[(ii)] on $D^1(R_0)$, $G_H^+< G_H^-.$
\end{itemize}
\end{prop}
\begin{proof}
Let $d_i \ge 2$ be the degree of each $H_i$, $1 \le i \le m$. By assumption $V^\pm=V^\pm_{R_H}$ and $H_i(V^+) \subset V^+$, $H^{-1}_i(V^-) \subset V^-$ for every $1 \le i \le m.$ Further there exist constants $m$ and $M$, such that on $V^+$
\begin{align}\label{mM1}
m|y|^{d_i}<|\pi_2 \circ H_i(x,y)|<M|y|^{d_i} 
\end{align}
and on $V^-$
\begin{align}\label{mM2}
m|x|^{d_i}<|\pi_2 \circ H_i^{-1}(x,y)|<M|x|^{d_i}.
\end{align}
Patching up (\ref{mM1}) and (\ref{mM2}) for every $i$, there exist constants $C_1$ and $C_2$ such that 
\begin{align}\label{C1C2}
\nonumber C_1|x|^d &\le |\pi_2 \circ H(x,y)|  \le C_2|x|^d \text{ on } V^- \text{ and }, \\
 C_1 |y|^d &\le |\pi_2 \circ H(x,y)|  \le C_2|y|^d \text{ on } V^+
\end{align}
where $d=d_m \hdots d_1.$ Suppose $H_1(x,y)=(b_1 y+c_1,p_1(y)-\de_1 x)$ then 
$$H_1^{-1}(x,y)=\Big(\de_1^{-1}(p_1\big(b_1^{-1}(x-c_1)\big)-y),b_1^{-1}(x-c_1)\Big).$$ 
Let $C=\max\{|b_1^{-1}(x-c_1)|: |x|<R_H\}.$ Choose $R_1>\max\{C,R_H\}$ such that for $|y|>R_1$, $$|\de_1^{-1}(p_1\big(b_1^{-1}(x-c_1)\big)-y)|> \max\{C, R_H\}$$ whenever $|x|<R_H$. Clearly, this means that 
\begin{align}\label{H1}
H_1^{-1}(x,y) \in V^-\text{ for }(x,y) \in D^2(R_1).
\end{align}
Thus from the assumptions on $R_H$ and (\ref{H1}) for every $n \ge 1$,  $$H^n(D^2(R_1)) \subset V^+ \text{ and } H^{-n}(D^2(R_1)) \subset V^-.$$ Let
\[ X(x,y)=\pi_1 \circ H_1^{-1}(x,y)=\de_1^{-1}(p_1\big(b_1^{-1}(x-c_1)\big)-y)\]
Now from (\ref{mM1}) and (\ref{mM2}) there exist constants $C_3$ and $C_4$ such that
\begin{align}\label{C3C4}
 C_3 |X(x,y)|^{dd_1^{-1}} \le |\pi_1 \circ H^{-1}(x,y)| \le C_4|X(x,y)|^{dd_1^{-1}}.
\end{align}
for $(x,y) \in D^2(R_1).$  So at the $n-$th iterate from (\ref{C1C2}) we have the following:
\begin{align}\label{eqn8.1}
\nonumber C_1^{1+d+ \cdots +d^{n-2}} |\pi_1 \circ H^{-1}(x,y)|^{d^{n-1}} &\le |\pi_1 \circ H^{-n}(x,y)| \le C_2^{1+d+ \cdots +d^{n-2}}|\pi_1 \circ H^{-1}(x,y)|^{d^{n-1}} \\
C_1^{1+d+ \cdots +d^{n-1}} |y|^{d^{n}} & \le |\pi_2 \circ H^{n}(x,y)| \le C_2^{1+d+ \cdots +d^{n-1}}|y|^{d^{n}}.
\end{align}
From (\ref{C3C4}) and (\ref{eqn8.1}) it follows that
\begin{align}\label{eqn8.2}
\nonumber C_1^{(d^{n-1}-1)/(d-1)}C_3^{d^{n-1}} |X(x,y)|^{d^nd_1^{-1}} &\le |\pi_1 \circ H^{-n}(x,y)| \le C_2^{(d^{n-1}-1)/(d-1)}C_4^{d^{n-1}} |X(x,y)|^{d^nd_1^{-1}} \\
C_1^{(d^n-1)(d-1)^{-1}} |y|^{d^{n}} & \le |\pi_2 \circ H^{n}(x,y)| \le C_2^{(d^n-1)(d-1)^{-1}}|y|^{d^{n}}.
\end{align}
Hence from (\ref{eqn8.2}) the constants $C_1$, $C_2$, $C_3$ and $C_4$ can be appropriately modified to obtain the following:
\begin{align}\label{final}
\nonumber \frac{d^{n-1}}{d-1}\log \ti{C}_1+d^{n}d_1^{-1}\log|X(x,y)|  &\le \log|\pi_1 \circ H^{-n}(x,y)| \le \frac{d^{n-1}}{d-1}\log \ti{C}_2 +{d^{n}d_1^{-1}}\log|X(x,y)|\\
\frac{d^{n-1}}{d-1}\log \ti{C}_3+{d^{n}}\log|y|  &\le \log|\pi_2 \circ H^{n}(x,y)| \le \frac{d^{n-1}}{d-1}\log \ti{C}_4+{d^{n}}\log|y|.
\end{align}
Note that on $V^+$, 
$$G_H^+=\lim_{n \to \infty} \frac{\log|\pi_2 \circ H^n(x,y)|}{d^n}$$ and on $V^-$, 
$$G_H^-=\lim_{n \to \infty} \frac{\log|\pi_1 \circ H^{-n}(x,y)|}{d^n}.$$
Dividing (\ref{final}) with $d^n$ and taking limit $n \to \infty$ it follows that
\begin{align}\label{eqn8.3}
 \frac{1}{d(d-1)} \log \ti{C}_1+\frac{1}{d} \log|X(x,y)| \le G_H^-(x,y) \le \frac{1}{d(d-1)} \log \ti{C}_2+\frac{1}{d_1} \log|X(x,y)|
 \end{align}
and
\begin{align}\label{eqn8.4} 
\frac{1}{d-1} \log \ti{C}_3+ \log|y| \le G_H^+(x,y) \le \frac{1}{d-1} \log \ti{C}_4+ \log|y|.
\end{align}
Since $(x,y) \in D^2(R_1)$, there exists $K>1$ such that
\[ |X(x,y)|=|\de_1^{-1}(p_1\big(b_1^{-1}(x-c_1)\big)-y)|< K |y|.\]
Thus from (\ref{eqn8.3}) and (\ref{eqn8.4}) there exists a real constant $M$ such that 
\[ G_H^--G_H^+< M+\frac{1-d_1}{d_1} \log|y|<0\]
for $|y|$ sufficiently large and $(x,y) \in D^2(R_1).$ Further modifying the choice of $R_1$ to $R_0^+$ (sufficiently large) it follows that $G_H^-< G_H^+$ on $D^2(R_0^+).$

\medskip\no Similarly there exists $R_0^-$ such that $G_H^+<G_H^-$ on $D^1(R_0^-).$ Now choose $R_0=\max\{R_0^+,R_0^-\}$ and this completes the proof.
\end{proof}
\begin{thm}\label{Green}
Let $H$ and $F$ be two H\'{e}non maps such that the Green current associated to $H$ and $F$ are same, i.e., $\mu_H=\mu_F.$ Then
\[ F \circ H=C_\eta \circ H \circ F=H \circ F \circ C_\eta^{\pm}.\] 
\end{thm}
\begin{proof}
Let $G_H$ and $G_F$ denote the Green function associated to $F$ and $H.$ Then 
\[ \mu_H=\Big(\frac{1}{2 \pi}dd^c G_H\Big)^2=\Big(\frac{1}{2 \pi}dd^c G_F\Big)^2=\mu_F.\]
By Theorem \textbf{1} in \cite{BedfordTaylor}, we have that $G_H=G_F.$   

\medskip\no Let
\[ H=H_m \circ \cdots \circ H_1 \text{ and }F=F_n \circ \cdots \circ F_1\]
where $H_j$ and $F_i$ are H\'{e}non maps of the form (\ref{1.2}). There exists $R>0$ such that if $V^\pm=V_R^\pm$ then
$$ H_i(V^+) \subset V^+, \; F_j(V^+) \subset V^+,\; H_i^{-1}(V^-) \subset V^- \text{ and } F_j^{-1}(V^-) \subset V^-$$ for every $1 \le i \le m$ and $1 \le j \le n.$ 

\medskip\no By Proposition \ref{Green-main} there exist $R_H^-$, $R_F^+$ and $R_F^-$ such that $G_H^+<G_H^-$ on $D^1(R_H^-)$, $G_F^+<G_F^-$ on $D^2(R_F^+)$ and $G_F^+<G_F^-$ on $D^1(R_F^-).$ Let $R_0=\max\{R_H^+,R_H^-,R_F^+,R_F^-\}.$ Since $$G_H=\max(G_H^+,G_H^-)\text{ and }G_F=\max(G_F^+,G_F^-),$$
it follows that on $D^2(R_0)$
\begin{align*}
G_F^+=G_F&=G_H=G_H^+,
\end{align*}
and on $D^1(R_0)$
\begin{align*}
G_F^-=G_F&=G_H=G_H^-.
\end{align*}
Also, $G_H^+$ and $G_H^-$ are harmonic on $V_R^+$ and $V_R^-$ respectively, hence
$G_H^+ = G_F^+ \text{ on } V_R^+$, and $ G_H^- = G_F^- \text{ on } V_R^-.$  

\medskip\no Now note that in the proof of Theorem \ref{relation}, we only use the fact that the positive Green functions of $F$ and $H$ agree on $V^+$ and negative Green functions of $F$ and $H$ agree in $V^-.$ Hence by using exactly the same arguments as in the proof of Theorem \ref{relation} it follows that there exists $\eta $ such that
\[  F\circ H=C_\eta \circ H \circ F=H \circ F \circ C_{\eta}^{\pm}\]
where $C_\eta(x,y)=(\eta x, \eta^{-1} y).$
\end{proof}
\begin{rem}\label{equivalent}
The fact that, Green measure is the equillibrium measure supported on the filled Julia set of a H\'{e}non map together with the proof of Theorem \ref{Green} implies that for two H\'{e}non maps $H$ and $F$, $\mu_H=\mu_F$ is equivalent to $G_H=G_F$ is equivalent to $K_H=K_F.$
\end{rem}
\begin{proof}[Proof of Theorem \ref{commute}]
By Theorem \ref{relation}, Theorem \ref{Green} and Remark \ref{equivalent} in either of these cases there exists a $C_\eta$ such that
\[F \circ H=C_\eta \circ  H \circ F=H \circ F \circ C_\eta^\pm.\]
Note that in this case the first diagonal element of $A=DF(H(0))DH(0)$ is always non--zero, forcing $\de_1=\ti{\de}_1.$ Thus
\[F \circ H=C_\eta \circ H \circ F=H \circ F \circ C_\eta.\]
From $ F \circ H=C_\eta \circ H \circ F$, by equating the first coordinate it follows that 
\begin{align*}
\eta\{q(y)-bx\}=p(y)-ax, \text{ i.e., } \eta q(y)=p(y) \text{ and } \eta  b=a.
\end{align*} 
Also, $ F^{-1} \circ H^{-1}=C_\eta \circ H^{-1} \circ F^{-1}$. By comparing the second coordinate we have
\[ b^{-1}(q(x)-y)=\eta^{-1}a^{-1}(p(x)-y),\text{ i.e., } \eta^{-1}a^{-1}=b^{-1} \text{ and }q(x)=p(x) .\]
This is not possible for $\eta \neq 1$, hence the proof.
\end{proof}
%\begin{proof}
%We will prove that for $H$ and $F$ in $\cal{H}$, if $\mu_H=\mu_F$ then $H=F.$
%
% By comparing appropriate coordinates as in Corollary \ref{commute}
%\begin{align*}
%\eta\{bx+q(y)\}&=ax+p(y), \text{ i.e., } \eta q(y)=p(y) \text{ and } \eta  b=a\\
%b^{-1}\{y-q(x)\}&=\eta^{-1}a^{-1}\{y-p(x)\},\text{ i.e., } \eta^{-1}a^{-1}=b^{-1} \text{ and }q(x)=p(x).
%\end{align*} 
%The above equations are compatible, if and only if $\eta=1$, thus we have $F=H.$ 
%\end{proof}
\section{Proof of Theorems \ref{improved_relation}, \ref{KHpm} and \ref{injective}}\label{3}
\no In this section, we will further improve Theorem \ref{relation}. We say that a H\'{e}non map $H$ can be expressed in \textit{normal form} if
\begin{align}\label{normal form}
H(z)=H_m\circ \cdots \circ H_1(z)
\end{align}
where each $H_i(x,y)=(y,p_i(y)-\delta_i x)$ such that degree of $p_i \ge 2$ and $\delta_i \neq 0.$
\begin{lem}\label{normal theorem}
Given a H\'{e}non map of the form (\ref{1.1}), say $H$, $$H^2=H \circ H$$ can be expressed in normal form.
\end{lem}
\begin{proof}
By assumption, $H=H_m \circ \cdots \circ H_1$ where $H_i$'s are H\'{e}non map of the form (\ref{1.2}), i.e.,
\[ H_i(x,y)=(b_i y+c_i,p_i(y)-\de_i x)\]
such that $\de_ib_i \neq 0$ and $d_i$, the degree of $p_i$ is greater than or equal to $2.$ Then 
\[ H^2=H_m \circ \cdots \circ H_1 \circ H_m \circ \cdots \circ H_1\] and $H^2$ is an even composition of H\'{e}non maps of the form (\ref{1.2}). Now
\[ H_{i+1} \circ H_{i}(x,y)=(b_{i+1}(p_i(y)-\de_i x)+c_{i+1}, p_{i+1}(p_i(y)-\de_i x))-\de_{i+1}b_i y-\de_{i+1}c_i).\]
Define $\ti{H}_i(x,y)=(y,\ti{p}_i(y)-\ti{\de}_i x)$ and $\ti{H}_{i+1}(x,y)=(y,\ti{p}_{i+1}(y)-\ti{\de}_{i+1} x)$ where
$$\ti{p}_i(x)=b_{i+1}p_i(y)+c_{i+1}, \; \ti{\de}_i=b_{i+1}\de_i$$
and
$$\ti{p}_{i+1}(x)=p_{i+1}(b_{i+1}^{-1}(x-c_{i+1}))-\de_{i+1}c_i, \; \ti{\de}_{i+1}=\de_{i+1}b_i.$$ Note that
\[ \ti{H}_{i+1} \circ \ti{H}_i(x,y)=H_{i+1} \circ H_i(x,y)\]
and hence the proof.
\end{proof}
\begin{lem}\label{non-normal theorem}
Let $H$ be a H\'{e}non map such that
\[ H=H_m \circ \cdots \circ H_1\]
where $m \ge 2$ and $H_i(x,y)=(b_i y,p_i(y)-\de_i x)$ for every $1 \le i \le m.$ Then $H$ can be expressed in normal form as a $m-$composition. 
\end{lem}
\begin{proof}
We will prove by induction on $m.$ For $m=2$, i.e., $H=H_2 \circ H_1$ the proof is same as Lemma \ref{normal theorem}, by taking $c_1=c_2=0.$ Assume that the statement is true for some $m \ge 2$, and prove for $m+1$, i.e., 
\[ H=H_{m+1} \circ \cdots \circ H_1.\]
consider $\ti{H}_{m+1}(x,y)=(y, p_{m+1}(y)-\de_{m+1} x)$ and $\ti{H}_m(x,y)=(b_{m}y, b_{m+1} p_m(y)-b_{m+1} \de_m x)$ then
\[ \ti{H}_{m+1}^{-1}\circ H= \ti{H}_m \circ H_{m-1}\circ \cdots \circ H_1.\]
By \textit{induction hypothesis}, $\ti{H}_{m+1}^{-1}\circ H$ can be expressed in normal form and thus the proof.
\end{proof}
\begin{lem}\label{0-fixed}
Let $H$ be a H\'{e}non map of the form (\ref{1.1}) which can be expressed in the normal form as:
\[ H=H_m \circ \cdots\circ H_1\]
where $H_i(x,y)=(y,p_i(y)-\delta_i x)$ and $p_i(0)=0$ for every $1 \le i \le m.$ Further if there exists $C_\eta$, such that
\[ C_\eta \circ H=H \circ C_\eta \text{ or }C_\eta \circ H=H \circ C_\eta^{-1}\]
for some $|\eta|=1.$ Then $\eta p_i(\eta y)=p_i(y)$ for every $1 \le i \le m.$
\end{lem}
\begin{proof}
The above statement is always true for $\eta=1.$ So assume $\eta \neq 1$. Then
\begin{align}\label{modification}
 (H_m^{-1} C_\eta H_m) \circ H_{m-1} \circ \cdots \circ H_1=H_{m-1} \circ \cdots \circ H_1 \circ C_\eta^{\pm}.
 \end{align}
Let $A_m=H_m^{-1} C_\eta H_m.$ Note that $A_m$ should be a linear map of the form $(ax+c,dx+ey+f)$, otherwise there will be an inconsistency in the highest degree of $y$ on both the sides of (\ref{modification}). Since 
\[ H_m^{-1}(x,y)=(\de_m^{-1}(p_m(x)-y),x),\] by computing $A_m$ we have that 
\[ A_m(x,y)=(\de_m^{-1}(p_m(\eta y)-\eta^{-1}p_m(y))+\eta^{-1} x, \eta y).\]
Thus $\de_m^{-1}(p_m(\eta y)-\eta^{-1}p_m(y))$ should be a constant (say $C_m$). Since $p_m(0)=0$, it follows that $C_m=0$, i.e., $\eta p_m(\eta y)=p_m(y).$ Also, (\ref{modification}) reduces to
\[ C_\eta^{-1} \circ H_{m-1} \circ \cdots \circ H_1=H_{m-1} \circ \cdots \circ H_1 \circ C_\eta^{\pm}.\]
Now by inductive argument the proof follows.
%Define$A_i$'s inductively for $1 \le i \le m-1$ as: 
%$$A_{i}=H_{i}^{-1}\circ A_{i+1}\circ H_i.$$ And an exact similar computation and argument gives
%\[ A_{i}=(\eta^{\pm} x+C_i,\eta^{\mp} y+C_{i+1})\] depending on the fact $m-i$ is odd or even.
%Since $A_1=C_\eta^{\pm}$, it follows that $C_1=0.$ But $$C_1=\de_1^{-1}(p_1(\eta^\pm y+C_2)-\eta^{\mp}p_1(y)) \text{ i.e., }
%\eta^\pm p_1(\eta^\pm y+C_2)=p_1(y).$$
%\medskip\no \textit{Claim:} $C=0.$
%
%\medskip \no Clearly if $m=1$, then $C=0$ as $A_1 \circ C_\eta^{\pm}=\text{Identity}.$ For $m \ge 2$
%\[ C_\eta \circ H_m \circ H_{m-1}=\ti{H}_{m} \circ \ti{H}_{m-1}\]
%where $$\ti{H}_m(x,y)=(y,\eta^{-1}p_m(\eta^{-1}y-\eta^{-1}\de_m x) \text{ and }\ti{H}_{m-1}(x,y)=(y,\eta p_{m-1}(y)-\eta \de_{m-1}x).$$
%Thus 
%\[ \ti{H}_{m} \circ \ti{H}_{m-1}\circ H_{m-2}\circ \cdots \circ H_1(x,y)=H_m \circ H_{m-1}\circ \hdots H_1(x,y) \circ C_\eta^{\pm}.\] Let
%\[ B_m=H_m^{-1} \circ \ti{H}_m \text{ and } B_m^{-1}=\ti{H}^{-1}_m \circ H_m.\]
\end{proof}
\begin{rem}\label{0-fixed rem}
Note that from the proof of Lemma \ref{0-fixed}, it follows that, for every $1 \le i \le m$
\[ C_\eta \circ H_i=H_i \circ C_\eta^{-1}.\]
\end{rem}
\begin{lem}\label{0-fixed 2}
Let $H$ be a H\'{e}non map such that $H(0)=0.$ Then $H$ can be represented as:
\[ H=H_{m} \circ \cdots \circ H_1\]
where $H_i(x,y)=(b_i y,p_i(y)-\delta_i x)$ and $p_i(0)=0$ for every $1 \le i \le m.$
\end{lem}
\begin{proof}
Recall from Corollary 2.3 in \cite{FM} the expression of a H\'{e}non map is unique upto composition by linear elementary maps. In particular, if
\[ H=H_m \circ \cdots \circ H_1=\ti{H}_{\ti{m}} \circ \cdots \circ \ti{H}_1\] where $H_i$ and $\ti{H}_i$ are H\'{e}non maps of the form (\ref{1.2}).
Then $m=\ti{m}$ and $H_i=s_i \circ \ti{H}_i$ or $H_i=\ti{H}_i \circ s_i$ where $s_i(x,y)=(ax+b,cy+d ).$ Hence, we can use induction on the number of H\'{e}non maps (here it is $m$) of the form (\ref{1.2}) composed to obtain the given H\'{e}non map.

\medskip\no
\textit{Initial case: }For $m=1$, since $H(0)=0$, it follows that $H_1=\ti{H}_1$ and the statement is true.

\medskip\no\textit{Induction Statement: }Suppose the statement is true for $m-1$, i.e., if $H$ is a H\'{e}non map such that $H(0)=0$ obtained by $(m-1)-$composition in normal form. Then $H$ can be expressed as:
\[ H=H_{m-1} \circ \cdots \circ H_1\]
where $H_i(x,y)=(b_i y,p_i(y)-\delta_i x)$ and $p_i(0)=0$ for every $1 \le i \le m-1.$

\medskip\no \textit{General case: } Thus
\[ H=\ti{H}_{m} \circ \cdots \circ \ti{H}_1\]
where $\ti{H}_i(x,y)=(b_i y+c_i,\ti{p}_i(y)-\delta_i x)$  for every $1 \le i \le m.$ Define, $p_1(x)=\ti{p}_1(x)-\ti{p}(0)$ and let 
$$H_1(x,y)=(b_1 y,p_1(y)-\de_1 x).$$ Now
\[ \ti{H}=H \circ H_1^{-1}=\ti{H}_{m} \circ \cdots \circ \ti{H}_2 \circ A \text{ where }A(x,y)=(x+c,y+d).\] Now redefine $\ti{H}_2=\ti{H}_2 \circ A.$ Then $\ti{H}$ is a H\'{e}non map such that $\ti{H}(0)=0$ obtained by \\ $(m-1)-$composition in normal form. By \textit{Induction hypothesis}, $\ti{H}$ can be expressed as:
\[ \ti{H}=H_{m} \circ \cdots \circ H_2\]
where $H_i(x,y)=(b_i y,p_i(y)-\delta_i x)$ and $p_i(0)=0$ for every $2 \le i \le m.$ Thus the proof.
\end{proof}
\begin{rem}\label{independent of bj}
Suppose $H$ is a H\'{e}non map such that $H(0)=0$ and 
\[ H=H_m \circ \cdots \circ H_1\]
where $m \ge 2$ and $H_i(x,y)=(b_i y, p_i(y)-\de_i x).$ From Lemma \ref{non-normal theorem} and Lemma \ref{0-fixed 2}, $H$ can be expressed in the normal form, i.e.,
\[ H=\ti{H}_m \circ \cdots \circ \ti{H}_1\]
where $\ti{H}_i(x,y)=(y, \ti{p}_i(y)-\ti{\de}_i x)$ for $1 \le i \le m.$ Note that for every $1 \le i \le m$, there exist non-zero constants $m_i$ and $M_i$ such that
\[ \ti{p}_i(y)=m_i p_i(M_i y).\] Hence, for some $|\eta|=1$, if $\eta \ti{p}_i(\eta y)=$ then $\eta p_i(\eta y)=p_i(y).$
\end{rem}
\no Now we can complete the proof of Theorem \ref{improved_relation}.
\begin{proof}[Proof of Theorem \ref{improved_relation}]
From Theorem \textbf{1.1} of the \cite{RigidityPaper}, we know that $F$ should be a polynomial map, in particular either $F$ or $F^{-1}$ is a H\'{e}non map of the form (\ref{1.1}). 
%
%\medskip\no If $F$ is a H\'{e}non map, then from Lemma \ref{normal theorem}, $F^2$ and $H^2$ can be expressed in \textit{normal form}, i.e.,
%\[ F^2=F_{2n} \circ \cdots \circ F_{1} \text{ and } H^2=H_{2m} \circ \cdots \circ H_1\]
%where $F_j(x,y)=(y,q_j(y)-\rho_j x)$ and $H_i(x,y)=(y,p_i(y)-\delta_i x)$ with degree $p_i$ and $q_j \ge 2$ and $\delta_i, \rho_j\neq 0$ for every $1 \le i \le 2m$ and $1 \le j \le 2n.$

\medskip\no By Theorem \ref{relation}, it follows that there exist $\eta_1$ and $\eta_2$ with $|\eta_i|=1$ for $i=1,2$ such that 
\begin{align}\label{eta_rel_1}
 F \circ H=C_{\eta_1} \circ H \circ F=  H \circ F\circ C_{\eta_1} \text{ or }F \circ H=C_{\eta_1} \circ H \circ F=  H \circ F\circ C_{\eta_1}^{-1}
 \end{align}
 and
 \begin{align}\label{eta_rel_2}
 F^2 \circ H^2=C_{\eta_2} \circ H^2 \circ F^2=  H^2 \circ F^2\circ C_{\eta_2} \text{ or }F^2 \circ H^2=C_{\eta_2} \circ H^2 \circ F^2=  H^2 \circ F^2\circ C_{\eta_2}^{-1}
 \end{align}
where $C_{\eta_i}(x,y)=(\eta_i x,\eta^{-1}_i y).$

\medskip\no\textit{Case 1: } If $H(0)=0.$

\medskip\no \textit{Subcase 1: }If $F(0) \neq 0$ then from (\ref{eta_rel_1})
\[ C_{\eta_1} \circ H \circ F(0)=H \circ F(0) \neq 0.\] Thus proving $C_{\eta_1}=\text{Identity}$ and $H^2 \circ F^2=F^2 \circ H^2.$

\medskip\no\textit{Subcase 2: } If $F(0)=0$ then by Lemma \ref{0-fixed 2}, the expression of $F$ and $H$ can be modified such that
\[ F=F_{n} \circ \cdots \circ F_{1} \text{ and } H=H_{m} \circ \cdots \circ H_1\]
where $F_j(x,y)=(\beta_j y,q_j(y)-\rho_j x)$ and $H_i(x,y)=(b_iy,p_i(y)-\delta_i x)$ with $p_i(0)=0$ and $q_j(0)=0$ and $b_i\delta_i, \beta_j\rho_j\neq 0$ for every $1 \le i \le m$ and $1 \le j \le n.$

\medskip\no From (\ref{eta_rel_1}) there exists $C_{\eta_1}$ such that
\[ C_{\eta_1} \circ H \circ F=H \circ F \circ C_{\eta_1}^{\pm}.\]
From Lemma \ref{non-normal theorem}, $H \circ F$ can be expressed in the normal form. Further from Remark \ref{0-fixed rem} and \ref{independent of bj}, if any $p_i$'s or $q_j$'s have a linear term then
\[ F \circ H=\pm H \circ F.\] If $F \circ H =H \circ F$ then clearly $F^2 \circ H^2 =H^2 \circ F^2.$ Otherwise, if $F \circ H =-H \circ F$, then $$F \circ H(x,y)=H \circ F(-x,-y)=-H \circ F(x,y).$$
Now 
\begin{align*}
 F \circ H \circ F \circ H &=H \circ F \circ H\circ F, \\
 F \circ (-F \circ H)\circ H&=H \circ (-H\circ F) \circ F,\\
 -F^2 \circ H^2=-H^2 \circ F^2, &\text{ i.e., } F^2 \circ H^2=H^2 \circ F^2.
 \end{align*}
 Suppose none of the $p_i$'s or $q_j$'s have a linear term then 
 \[ DF_i(0)=\begin{pmatrix}
 0 & \beta_j \\
 -\rho_j & 0
 \end{pmatrix} \text{ and }
 DH_j(0)=\begin{pmatrix}
  0 & b_i \\
 -\de_i & 0
\end{pmatrix}
.\]
Thus $DF^2(0)$ and $DH^2(0)$ are diagonal matrices. Since $H^2(0)=0$ and $F^2(0)=0$,
\[ D(F^2 \circ H^2)(0)=DF^2(0). DH^2(0)=DH^2(0).DF^2(0)=D(H^2 \circ F^2)(0).\]
But from (\ref{eta_rel_2})
\[  D(F^2 \circ H^2)(0)=C_{\eta_2}. D(H^2 \circ F^2)(0)\]
thus proving $\eta_2=1.$ Hence $F^2 \circ H^2=H^2 \circ F^2.$

\medskip\no\textit{Case 2: } If $H(0) \neq 0$, then consider any fixed point of the H\'{e}non map, say $p=(p_1,p_2).$ Let $A_p$ be the affine map translating $0$ to $p$, i.e.,
\[ A_p(x,y)=(x+p_1,y+p_2) \text{ and }A_p^{-1}(x,y)=(x-p_1,y-p_2).\]
Define $\ti{H}=A_p^{-1} \circ H \circ A_p$ and $\ti{F}=A_p^{-1} \circ F \circ A_p.$ The both $\ti{H}$ and $\ti{F}$ are H\'{e}non maps and $\ti{H}(0)=0.$

\medskip\no \textit{Claim: }$K_{\ti{H}}^\pm=A_p^{-1}(K_H^\pm)$ and $\ti{F}(K_{\ti{H}}^\pm)=K_{\ti{H}}^\pm.$

\medskip\no For $z \in A_p^{-1}(K_H^+)$, $\ti{H}^n(z)$ is bounded for every $n$, thus $A_p^{-1}(K_H^+) \subset K_{\ti{H}}^+.$  Similarly, for $z \in A_p(K_{\ti{H}}^+)$, $H^n(z)$ is bounded for every $n$ proving $A_p(K_{\ti{H}}^+) \subset K_{H}^+.$ By an exactly similar argument for $H^{-1}$ and $K_H^-$ we have
 $$K_{\ti{H}}^-=A_p^{-1}(K_H^-), \text{ i.e., } K_{\ti{H}}^\pm=A_p^{-1}(K_H^\pm). $$  
Further as $F(K_H^\pm)=K_H^\pm$, $$\ti{F}(K_{\ti{H}}^\pm)=\ti{F}(A_p^{-1}(K_H^\pm))=A_p^{-1}(K_H^\pm)=K_{\ti{H}}^\pm.$$
Now by \textit{Case 1},
\[ \ti{F}^2 \circ \ti{H}^2=\ti{H}^2 \circ \ti{F}^2, \text{ or } F^2 \circ H^2=H^2 \circ F^2.\] 

\medskip \no Finally, if $F^{-1}$ is a H\'{e}non map then by the above arguments we have
\[ F^{-2} \circ H^2=H^2 \circ F^{-2}, \text{ i.e., } H^2 \circ F^2=F^2 \circ H^2.\]
\end{proof}
\no In the following example we show that Theorem \ref{improved_relation} is optimal, in the sense there exist H\'{e}non maps $H$ and $F$ such that $F(K_H^\pm)=K_H^\pm$ but they do no commute.
\begin{exam}\label{opti-com}
Let $H(x,y)=(y,y^2-x)$ and $F(x,y)=(\omega y, (\omega y)^2-\omega^2 x)$ where $\omega$ is the cube root of unity. Note that
\[ F(x,y)=C_\omega \circ H(x,y)=H(x,y) \circ C_{\omega^2}.\] Hence $F$ is a H\'{e}non map such that $F(K_H^\pm)=K_H^\pm.$ Now
\[ F \circ H=C_\omega \circ H^2=H^2 \circ C_\omega=H \circ F \circ C_{\omega^2}=C_{\omega^2} \circ H \circ F.\]
thus $H$ and $F$ do not commute.
\end{exam}
\no Finally, we prove Theorem \ref{KHpm} and \ref{injective} using Theorem \ref{improved_relation}.
\begin{proof}[Proof of Theorem \ref{KHpm}]
Recall that, by Remark \ref{equivalent}, it is enough to prove for the case $\mu_H=\mu_F.$ Further from Theorem \ref{Green}, it follows that 
 \[F \circ H=C_\eta \circ H \circ F=  H \circ F\circ C_\eta^{\pm}\] where 
 $C_\eta(x,y)=(\eta x, \eta^{-1} y).$ Now by applying exactly similar argument as in the proof of Theorem \ref{improved_relation}, it follows that
 \[ H^2 \circ F^2=F^2 \circ H^2.\]
 For $z \in K_H^+$
 \[ H^{2n} \circ F^2(z)=F^2 \circ H^{2n}(z) \text{ is bounded }.\] Thus $F^2(z) \in K_H^\pm$, or $F^2(K_H^+) \subset K_H^+.$ Similarly $F^{-2}(K_H^+)\subset K_H^+$. Also, applying the same arguments with $H^{-2n}$ it follows that $F^2(K_H^-)=K_H^-.$ Hence
 \[ F^2(K_H^\pm)=K_H^\pm \text { and } H^2(K_F^\pm)=K_F^\pm.\]
 Now for $z \in K_H^+$, $F^{2n}(z) \in V_R \cup V_R^-$ and $z \in K_F^+$, $H^{2n}(z) \in V_R \cup V_R^-$ proving $$K_H^+ \subset K_{F^2}^+=K_F^+ \text { and } K_F^+ \subset K_{H^2}^+=K_H^+.$$ By a similar argument for $H^{-1}$ and $F^{-1}$ it follows that $K_H^-=K_F^-$. Hence $ K_F^\pm=K_H^\pm$ and $G_H^\pm=G_F^\pm.$
\end{proof}
%\begin{thm}
%Let $\mathcal{H}$ denote the space of finite compositions of generalized H\'{e}non maps then the mapping $H \to \mu_H$ is locally injective on $\mathcal{H}.$
%\end{thm}
\begin{proof}[Proof of Theorem \ref{injective}]
Again, by Remark \ref{equivalent} it is sufficient to prove the statement only for the Green measure of $H$. Suppose $\{H_n\}$ is any sequence of H\'{e}non maps such that $\mu_{H_n}=\mu_H$ and $H_n \neq H.$ To prove the result, it is enough to prove that no subsequence of $\{H_n\}$ converges to $H$ in the topology of uniform convergence over compact subsets of $\mbb C^2$. In other words, there exist a neighbourhood $U_H$ around $H$ in $\cal{H}$ with the aforementioned topology, such that the map $H \to \mu_H$ is injective on $U_H.$ 

\medskip\no We will prove the result by contradiction. Suppose not, i.e., there exists a subsequence $\{H_{n_k}\}$ of $\{H_n\}$ such that 
$H_{n_k} \to H$ as $k \to \infty.$ With abuse of notation, we denote $\{H_{n_k}\}$ by $\{H_n\}.$ 

\medskip\no\textit{Claim: }$A_n=H_n^{-1} \circ H$ should be linear maps for $n-$sufficiently large.

\medskip\no Let $d_n$ be the degree of $H_n$ and $d$ the degree of $H.$ By Theorem \ref{KHpm}, $G_{H_n}^\pm=G_H^\pm$ and $K_{H_n}^{\pm}=K_H^\pm.$ Fix a $z \notin K_H^+$,
\[ G_H^+(z)=d^{-1}G_H^+(H(z))=d_nd^{-1}G_H^+(A_n(z)).\]
By assumption $A_n(z) \to z$ as $n \to \infty$ and $G_H^+(A_n(z)) \to G_H^+(z).$ Thus $d_nd^{-1} \to 1$ as $n \to \infty$, i.e., $d_n=d$ for $n-$sufficiently large. Now if $A_n$'s are non--linear then either $d_n<d$ or $d<d_n$, which is not true. Hence $A_n$'s are linear and by Theorem 1.1 from \cite{RigidityPaper}
\begin{align}\label{linear form}
A_n(x,y)=(\ti{a}_n x+\ti{b}_n,\ti{c}_n y+\ti{d}_n).
\end{align}
By Theorem \ref{improved_relation},
$ H_n^2 \circ H^2=H^2 \circ H_n^2$, hence  
\begin{align}\label{C_n}
C_n^{-1} \circ H^{2}\circ C_n=H^{2}
\end{align}
where $C_n=H_n^{2} \circ H^{-2}$ and $k \ge 1.$ Note that by an argument similar to $A_n$'s, $C_n$'s are also linear maps of the form \ref{linear form}. Let 
\[ C_n(x,y)=(a_n x+b_n, c_n y+d_n).\]
\no \textit{Case 1: } Suppose $H(0)=0.$ Note that the number of fixed point of $H^2$ should be finite. Let $S_{H^2}$ denote the set of fixed points of $H^2$, i.e.,
\[ S_{H^2}=\{ p \in \mbb C^2: H^2(p)=p\} \text{ and } \# S_{H^2}=N_0 \ge 1.\]
From (\ref{C_n}) it follows that $C_n(0) \in S_{H^2}$, hence $(b_n,d_n) \in S_{H^2}.$ Thus there are only finitely many choice for $b_n$'s and $d_n$'s. Now by Lemma \ref{normal theorem} and Lemma \ref{0-fixed 2}
\[ H^2=\ti{H}_m \circ \cdots \circ \ti{H}_1\]
such that $m \ge 2$ and $\ti{H}_i(x,y)=(y,p_i(y)-\de_i x)$ where $p_i(0)=0$  and degree of $p_i$, say $d_i \ge 2$ for every $1 \le i \le m.$ Note that
\begin{align}\label{pi_2}
\pi_2 \circ \ti{H}_m^{-1} \circ C_n \circ \ti{H}_m=a_n y+b_n.
\end{align}
Hence $\ti{H}_m^{-1} \circ C_n \circ \ti{H}_m$ should be a linear map of the (\ref{linear form}), otherwise there will be an inconsistency in the degree of $y$ in both the sides. This means
\begin{align}\label{pi_1}
 \pi_1 \circ \ti{H}_m^{-1} \circ C_n \circ \ti{H}_m=\de_m^{-1}(p_m(a_n y+b_n)-c_np_m(y)-d_n)+c_n x=c_n x+d'_n
\end{align}
for some $d'_n \in \mbb C.$ Thus from (\ref{pi_2}) and (\ref{pi_1}) it follows that
\[ a_n^{d_m}=c_n \text{ and } \ti{C}_n(x,y)=\ti{H}_m^{-1} \circ C_n \circ \ti{H}_m(x,y)=(c_n x+d_n', a_n y+b_n).\]
By a similar argument for $\ti{H}_{m-1} \circ \ti{C}_n \circ \ti{H}_{m-1}$ it follows that $c_n^{d_{m-1}}=a_n.$ Thus both $a_n$'s and $c_n$'s should be $(d_md_{m-1}-1)-$th root of unity and there can be only a finitely many choice for both $a_n$'s and $c_n$'s. So there are only finitely many possible choice for the elements in the sequence $\{C_n\}$ and hence for $H_n^2$'s. Now as $$H_n^2=C_n H^2 \to H^2 \text{ as }n \to \infty,$$ it follows that $H_n^2=H^2$ for $n$ sufficiently large (say $n_0$). Thus $H_n$'s should be a square root of $H^2$ for all $n \ge n_0.$ By Theorem 4.1 in \cite{RootBuzzard}, there can only be a finitely many choice for $H_n$'s again. Hence, $H_n=H$ for $n \ge \ti{n}_0$ (sufficiently large). This completes the proof for \textit{Case 1}.

\medskip\no \textit{Case 2: }Suppose $H(0) \neq 0$. Let $p_0$ be a fixed point of $H$, consider $\ti{H}=A_{p_0}^{-1} \circ H \circ A_{p_0}$ where $A_{p_0}(z)=z+p_0$, for $z \in \mbb C^2.$ By assumption $H_n \to H$, hence
$$A_{p_0}^{-1} \circ H_n \circ A_{p_0} \to \ti{H}$$ as $n \to \infty.$ Now by \textit{Case 1}, $A_{p_0}^{-1} \circ H_n \circ A_{p_0} = \ti{H}$, i.e., $H_n=H$ for $n$ sufficiently large. This completes the proof.
\end{proof}


\begin{thebibliography}{10}
  \bibitem{Be}
A. F. Beardon, \emph{Symmetries of {J}ulia sets}, Math. Intelligencer \textbf{18}
  (1996), no. 1, 43--44. \MR{1381578}
  
  \bibitem{BS}
E. Bedford and J. Smillie, \emph{Polynomial diffeomorphisms of {$\mathbb{C}^2$}: currents, equilibrium measure and hyperbolicity}, Invent. Math. \textbf{103} (1991), no.~1, 69--99. \MR{1079840}

\bibitem{BS1}
E. Bedford and J. Smillie, \emph{Polynomial diffeomorphisms of {$\mathbb{C}^2$}. {III}.
  {E}rgodicity, exponents and entropy of the equilibrium measure}, Math. Ann.
  \textbf{294} (1992), no.~3, 395--420. \MR{1188127}
  
  \bibitem{BedfordTaylor}E. Bedford and B.~A. Taylor, \emph{Uniqueness for the complex
  {M}onge-{A}mp\`ere equation for functions of logarithmic growth}, Indiana
  Univ. Math. J. \textbf{38} (1989), no.~2, 455--469. \MR{997391}
  
\bibitem {RigidityPaper} S. Bera, R. Pal, and K. Verma, \emph{A rigidity theorem for
  {H}\'{e}non maps}, arXiv preprint arXiv:1806.08189 (2018), (to appear) European Journal of Mathematics.
  
\bibitem{RootBuzzard}G. T. Buzzard and J. E. Fornaess, \emph{Compositional roots of
  {H}\'{e}non maps}, Geometric complex analysis ({H}ayama, 1995), World Sci.
  Publ., River Edge, NJ, 1996, 67--73. \MR{1453590}
 
   \bibitem{Dinh-Sibony}
T. C. Dinh and N. Sibony, \emph{Rigidity of {J}ulia sets for
  {H}\'{e}non type maps}, J. Mod. Dyn. \textbf{8} (2014), no.~3-4, 499--548.
  \MR{3345839}
  
    \bibitem{DF}
R. Dujardin and C. Favre, \emph{The dynamical {M}anin-{M}umford problem for
  plane polynomial automorphisms}, J. Eur. Math. Soc. (JEMS) \textbf{19}
  (2017), no.~11, 3421--3465. \MR{3713045}
  
  \bibitem{FM}S. Friedland and J. Milnor, \emph{Dynamical properties of plane
  polynomial automorphisms}, Ergodic Theory Dynam. Systems \textbf{9} (1989),
  no.~1, 67--99. \MR{991490}
 
 \bibitem{Hu-OV} J. H. Hubbard and R. W. Oberste-Vorth, \emph{H\'{e}non mappings in the
  complex domain. {I}. {T}he global topology of dynamical space}, Inst. Hautes
  \'{E}tudes Sci. Publ. Math. (1994), no.~79, 5--46. \MR{1307296} 
  
\bibitem {Lamy} S. Lamy, \emph{L'alternative de {T}its pour {${\rm Aut}[\mathbb{C}^2]$}}, J. Algebra \textbf{239} (2001), no.~2, 413--437. \MR{1832900}

\end{thebibliography}
\end{document}